\documentclass{amsart}
\usepackage[utf8]{inputenc}

\usepackage{amsmath,amssymb,mathrsfs,amsthm}
\usepackage{mathtools}
\usepackage[usenames]{color}
\usepackage{color} 
\usepackage[all]{xy}
\usepackage{graphicx}
\usepackage{wrapfig}
\usepackage{tabularx}
\usepackage{supertabular}
\usepackage{setspace} 
\usepackage{enumerate} 
\usepackage{comment}
\usepackage{tikz}
\usepackage{pgfplots} 
\usetikzlibrary{intersections, calc}
\usetikzlibrary{patterns}
\usepackage{tikz-3dplot}
\usetikzlibrary{arrows,decorations.markings}

\newcommand{\bs}[1]{\mbox{$\boldsymbol{#1}$}}
\newcommand{\sgn}{\mathrm{sgn}}

\newtheorem{theorem}{Theorem}[section]
\newtheorem{mainthm}{Theorem}

\newtheorem{lemma}[theorem]{Lemma}
\newtheorem{proposition}[theorem]{Proposition}

\newtheorem{definition}[theorem]{Definition}
\newtheorem{corollary}[theorem]{Corollary}

\title[Combinatorial Integration]
{A combinatorial integration on the Cantor dust}

\author[T. MARUYAMA]{Takashi MARUYAMA}
\address{Department of Engineering, Stanford University, 353 Jane Stanford Way Stanford CA 94305, USA}
\email{mtakashi279@cs.stanford.edu, 49takashi@gmail.com}

\author[T. SETO]{Tatsuki SETO}
\address{General Education and Research Center, Meiji Pharmaceutical University, 
2-522-1 Noshio, Kiyose-shi, Tokyo, Japan}
\email{tatsukis@my-pharm.ac.jp}


\newcommand{\tatsuki}[1]{{{\textcolor{red}{[Tatsuki: #1]}}}}

\begin{document}

\begin{abstract}

In this paper, we generalize the Cantor function to $2$-dimensional cubes and construct a cyclic $2$-cocycle on the Cantor dust. 
This cocycle is non-trivial on the pullback of the smooth functions on the $2$-dimensional torus with the generalized Cantor function while it vanishes on the Lipschitz functions on the Cantor dust. 
The cocycle is calculated through the integration of $2$-forms on the torus by using a combinatorial Fredholm module.


\end{abstract}

\maketitle

\section*{Introduction} 
Cyclic cohomology for algebras \cite{MR823176} is a fundamental tool to study noncommutative geometry. One of its application is the study of fractal sets to which powerful tools such as de Rham homology (on smooth manifolds) cannot be applied. $K$-homology, one of other homotopy invariant theories, of fractal sets is also studied extensively. For some class of fractal sets such as the Cantor set and the Cantor dust, the $K^{0}$ homology groups are isomorphic to $\displaystyle \prod^{\infty} \mathbb{Z}$, which is known as the Baer-Specker group \cite{MR1545974}. One feature of the Baer-Specker group is that the group does not admit basis.  This feature makes the study of $K^{0}$ groups for such fractal sets somewhat intractable because local topological features of the spaces cannot induce local algebraic structures. Cyclic cohomology on the one hand is expected to be favorable in these cases because it can be characterized as a ``linearization" of $K$-homology through the Chern character. 

There is a study \cite{talk:Moriyoshi2013} by H. Moriyoshi and T. Natsume which presented a variant of the Riemann-Stieltjes integration on the middle third Cantor set. 
Let $CS = \displaystyle \bigcap_{n=0}^{\infty}I_{n}$ be the 
middle third Cantor set, where 
$I_{0} = [0,1]$, $I_{1} = \left[ 0 , \dfrac{1}{3} \right]$, 
$I_{2} = \left[ 0 , \dfrac{1}{3^{2}} \right] \cup \left[ \dfrac{2}{3^{2}} , \dfrac{3}{3^{2}} \right] \cup \left[  \dfrac{6}{3^{2}} , \dfrac{7}{3^{2}} \right] \cup \left[ \dfrac{8}{3^{2}} , 1 \right]$, \dots.   
We also denote by $(H_{n} , F_{n})$ the Fredholm module on $I_{n}$ 
which is defined by the direct sum of Connes' Fredholm module on intervals. 
In \cite{talk:Moriyoshi2013}, a new class of algebra is introduced:
an algebra $\mathcal{P} = c^{\ast}BV(S^{1})$ defined by the pull-back of the bounded variation class on a unit circle $BV(S^{1})$ with the canonical Cantor function $c$. 
They defined a functional $\psi (f,g) = \displaystyle \lim_{n \to \infty}\psi_{n}(f,g) =  \lim_{n \to \infty} \mathrm{Tr}(\epsilon f [F_{n} , g]F_{n})$ on $\mathcal{P}$. 
The existence of the limit 
can be proved by showing that the cocycle is reduced to the Riemann-Stieltjes integration. 
A key ingredient for the proof is the Cantor function (and its intermediate functions that converge into the Cantor function); 
the function is used to map the iterated function system of Cantor sets onto $2$-fold subdivisions of the unit interval on which the Riemann-Stieltjes integration is defined. 
Because the Cantor function remains surjective onto the unit interval when the function is restricted to the Cantor set, the iterated function system for the Cantor set may be seen as a variant of subdivisions on $I$. 

In the present paper, by following the idea mentioned above, 
we construct a new cyclic $2$-cocycle on the Cantor dust $CD$ and show that the cocycle is not trivial. 
In order to construct the cocycle, 
we define a sequence of functionals 
$\phi_{n}$ (see Definition \ref{def:phi_n}),  
that is a generalization of a functional 
$\psi_{n}$ on $CS$,  
by using a combinatorial Fredholm module on squares defined by 
the authors \cite{arXiv:1912.05832} 
and 
use the product  of the Cantor function. 
The function   induces the map $\bs{c}$ between  
the Cantor dust and $2$-torus $\mathbb{T}^{2}$: 

\begin{mainthm}[see Theorem \ref{thm:index_thm}]
Let $C^{\infty}(\mathbb{T}^2)$ be the smooth functions on the torus $\mathbb{T}^{2}$. For $f =\bs{c}^{*}(\tilde{f}), g = \bs{c}^{*}(\tilde{g}), h = \bs{c}^{*}(\tilde{h}) \in \bs{c}^{*}C^{\infty}(\mathbb{T}^{2})$, 
we have 
\[
\lim_{n \to \infty} 
\phi_{n}(f,g,h) = 2\int_{\mathbb{T}^{2}} \tilde{f}d\tilde{g} \wedge d\tilde{h}.
\]
Moreover, the limit only depends on $f,g,h \in \bs{c}^{*}C^{\infty}(\mathbb{T}^{2})$. 
Therefore, the functional $\phi$ defined by the limit is a cyclic $2$-cocycle on 
$\bs{c}^{*}C^{\infty}(\mathbb{T}^{2})$. 
\end{mainthm}

We can take the value $\phi_{n}(f,g,h)$ for 
Lipschitz functions 
$f,g,h \in C^{\mathrm{Lip}}(CD)$ 
by the definition of $\phi_{n}$.  
However, we have $\displaystyle \lim_{n \to \infty}\phi_{n}(f,g,h) = 0$
for any 
$f,g,h \in C^{\mathrm{Lip}}(CD)$ (see Proposition \ref{prop:Lipschitz}).  
Thus the algebra  $\bs{c}^{*}C^{\infty}(\mathbb{T}^{2})$ is far different from  $C^{\mathrm{Lip}}(CD)$. 

All the main results to be shown also hold for the Sierpinski carpet.  $\mathbb{T}^{2}$ can be also replaced with a $2$-dimensional sphere. 
In order to generalize our main theorem to general dimension, 
we need a general method to prove the existence of a cyclic cocycle.  
We will leave the study for future work.





%

\section{A combinatorial Fredholm module on squares}
\label{sec:combfredmod}

In this section, we review a combinatorial Fredholm module on squares constructed by the authors \cite{arXiv:1912.05832}.  
Let $\gamma \subset \mathbb{R}^{2}$ be a square of dimension $2$ 
and $V =  \{v_{0} , \ v_{1}  , \ v_{2}  , \ v_{3} \}$ 
the set of vertices of $\gamma$;  
see the following figure for the numbering of the vertices.  

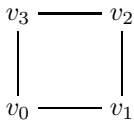
\begin{figure}[h] 
\label{fig:verticies} 
\[ \xymatrix{
v_{3} & v_{2} \ar@{-}[l] \\ 
v_{0} \ar@{-}[r] \ar@{-}[u] & v_{1} \ar@{-}[u]
} 
\] 
\caption{Numbering of the vertices} 
\end{figure}  

Set 
$V_{0} = \{ v_{0}, \ v_{2} \}$ and  
$V_{1} = \{ v_{1}, \ v_{3}\}$, 
so we have $V = V_{0} \cup V_{1}$. 
Set 
\begin{align*}
\mathcal{H}^{+} &= \ell^{2}(V_{0}) 
	= \ell^{2}(v_{0}) \oplus \ell^{2}(v_{2})  , & 
\mathcal{H}^{-} &= \ell^{2}(V_{1}) 
	= \ell^{2}(v_{1}) \oplus \ell^{2}(v_{3})  
\end{align*} 
and 
$\mathcal{H} = \mathcal{H}^{+} \oplus \mathcal{H}^{-}$. 
The vector space 
$\mathcal{H} (\cong \mathbb{C}^{4})$  is 
a Hilbert space of 
dimension $4$ with an inner product 
\[
\langle f , g \rangle 
= \sum_{i = 0}^{3} f(v_{i})\overline{g(v_{i})}. 
\]
We assume that $\mathcal{H}$ is $\mathbb{Z}_{2}$-graded 
with the grading $\epsilon = \pm 1$ on $\mathcal{H}^{\pm}$, respectively.  
The $C^{\ast}$-algebra 
$C(V)$  
of continuous functions on $V$   
acts on $\mathcal{H}$ by multiplication:  
\[
\rho(f) = (f(v_{0}) \oplus f(v_{2}) )    
	\oplus (f(v_{1}) \oplus f(v_{3})) .  
\] 

Set
$U = \dfrac{1}{\sqrt{2}}
\begin{bmatrix} 1 & -1 \\ 1 & 1 \end{bmatrix}$ 
and 
$F = 
\begin{bmatrix} & U^{\ast} \\ U & \end{bmatrix} 
=
\dfrac{1}{\sqrt{2}}\begin{bmatrix}
 & & 1 & 1 \\ 
 & & -1 & 1 \\ 
1 & -1 & & \\ 
1 & 1 & &  
\end{bmatrix}$. 
Then $F$ is a bounded operator on $\mathcal{H}$ and 
we have $F\epsilon + \epsilon F = O$. 
So $(\mathcal{H} , F)$ is the Fredholm module on $C(V)$. 

Set $I = [0,1] \times [0,1]$ and let 
$f_{s} : I \to I$ 
($s=1, \dots , N$) 
be similitudes.  
Denote by 
\[
r_{s} = \dfrac{\|f_{s}(x) - f_{s}(y)\|_{\mathbb{R}^{n}}}{\| x - y \|_{\mathbb{R}^{n}}} \;\,
(<1)
\quad (x \neq y) 
\]  
the similarity ratio of $f_{s}$. 
An iterated function system (IFS)
$(I, S = \{ 1, \dots , N \}, \{ f_{s} \}_{s \in S})$ 
defines the unique non-empty compact set 
$K = K(\gamma_{n}, S = \{ 1, \dots , N \}, \{ f_{s} \}_{s \in S})$ 
called the self-similar set 
such that 
$K = \bigcup_{s=1}^{N} f_{s}(K)$. 
Denote by $\dim_{S}(K)$ the similarity dimension of $K$, 
that is, the number $s$ that satisfies 
\[
\sum_{s=1}^{N}r_{s}^{s} = 1. 
\]
If an IFS
$(I, S, \{ f_{s} \}_{s \in S})$ 
satisfies 
the open set condition, we have 
$\dim_{H}(K) = \dim_{S}(K)$, 
where we denote by $\dim_{H}(K)$ 
the Hausdorff dimension of $K$. 

Set $f_{\bs{s}} = f_{s_{1}} \circ \dots \circ f_{s_{j}}$ for 
$\bs{s} = (s_{1}, \dots ,  s_{j}) \in S^{\infty} = \bigcup_{j=0}^{\infty} S^{\times j}$ and 
$f_{\emptyset} = \mathrm{id}$. 
For the simplicity, 
we will denote by $i$  
the vertex $f_{\bs{s}} (v_{i})$ 
of 
a square $f_{\bs{s}} (I)$. 
We also denote 
by $V_{\bs{s}}$ the vertices of a square $f_{\bs{s}} (I)$. 
Denote by $e_{\bs{s}}$ the length of edge of $f_{\bs{s}} (I)$, 
which equals $\prod_{k=1}^{j}r_{s_{k}}$.  
As introduced above, 
we set the Hilbert space $\mathcal{H}_{\bs{s}} = \ell^{2}(V_{\bs{s}})$ 
on $f_{\bs{s}} (I)$ of the length $e_{\bs{s}}$, 
which splits the positive part
$\mathcal{H}_{\bs{s}}^{+}$ 
and the negative part 
$\mathcal{H}_{\bs{s}}^{-}$. 
Taking direct sum on squares, we set as follows: 
\[ 
\mathcal{H}_{n} = \bigoplus_{j=1}^{n} \bigoplus_{\bs{s} \in S^{\times j}}\mathcal{H}_{\bs{s}}, 
\quad 
F_{n} = \bigoplus_{j=1}^{n} \bigoplus_{\bs{s} \in S^{\times j}} F
. 
\]  
The $C^{\ast}$-algebra $C(K)$ acts on $\mathcal{H}_{n}$ by multiplication:  
\[
\rho_{n}(f) =  \bigoplus_{j=1}^{n} \bigoplus_{\bs{s} \in S^{\times j}} \rho(f). 
\]


\section{Review on the Cantor dust}
\label{sec:cantordust}
We introduce an IFS of the Cantor dust and $2$-dimensional analogue of the Cantor function defined on $1$-dimensional interval. 
The IFS of the Cantor dust is our main interest in the paper. 
The analogue of the Cantor function plays a crucial role 
in construction of cyclic cocycle in Section \ref{sec:seccocycle}. 
Let $I = [0,1]\times[0,1]$. 
The IFS of the Cantor dust is defined as a set of functions $\{f_{s} : I \rightarrow \mathbb{R}^{2} \}_{s=1,2,3,4}$ described as follows:
\begin{align*}
    f_{1}(x) = \frac{1}{3}x, \ 
    f_{2}(x) = \frac{1}{3}x + \dfrac{1}{3}
    \begin{bmatrix}
        0 \\
        2
    \end{bmatrix},\ 
    f_{3}(x) = \frac{1}{3}x + \dfrac{1}{3}
    \begin{bmatrix}
           2 \\
           0
    \end{bmatrix},\ 
    f_{4}(x) = \frac{1}{3}x + \dfrac{1}{3} 
    \begin{bmatrix}
           2 \\
           2
    \end{bmatrix}.   
\end{align*}
This IFS induces a unique 
non-empty compact set $CD$ in $I$ such that 
$\displaystyle CD = \bigcup_{i=1}^{4} f_{s}(CD)$ holds; $CD$ is called the Cantor dust. 

Let $c_{0} = x: [0,1] \rightarrow \mathbb{R}$. We define in an inductive manner a sequence of $\mathbb{R}$-valued continuous functions defined on the unit interval $\{ c_{n} \}_{n \in \mathbb{N}}$ as follows:
\begin{equation*}
  c_{n+1}(x)=\begin{cases}
    \frac{1}{2}c_{n}(3x), & 0 \leq x \leq \frac{1}{3},\\
    \frac{1}{2}, & \frac{1}{3} \leq x \leq \frac{2}{3},\\
    \frac{1}{2}c_{n}(3x-2) + \frac{1}{2}, & \frac{2}{3} \leq x \leq 1.
  \end{cases}
\end{equation*}
The limit of $\{c_{n}\}_{n \in \mathbb{N}}$ exists and is called the Cantor function. 
We then generalize the construction to $2$-dimensional case by taking the pair of the two identical sequences:
\[
\bs{c}_{n} = (c_{n}, c_{n}): I \rightarrow \mathbb{R}^{2}.
\] 
The limit of the sequence $\{ \bs{c}_{n} \}$ also exists and 
equals the product of the Cantor functions. 
We call the limit $\displaystyle \lim_{n \to \infty} \bs{c}_{n}$ 
\textit{$2$-dimensional Cantor dust function}. 
We note that the construction of the $2$-dimensional Cantor dust function can be generalized to higher dimensional case.




\section{A combinatorial integration}
\label{sec:seccocycle}
\subsection{Approximating combinatorial integration}
\label{subsec:appcocycle}

We apply the construction of the combinatorial Fredholm module to the IFS of the Cantor dust $(I, S=\{1,2,3,4\}, \{ f_{s} \}_{s \in S})$ defined in Section \ref{sec:cantordust}. 
In order to construct our cyclic cocycle,  
we need the multiplication operator $\rho (f)$ 
and the commutator $[F,f]$ 
on $\mathcal{H}_{\bs{s}}$ 
for $f \in C(CD)$. 
$\rho (f)$ 
is even, so 
we can express 
$\rho(f) = \begin{bmatrix} 
f^{+} & \\ & f^{-} 
\end{bmatrix}$. 
We can write down the commutator $[F,f]$ 
as follows: 
\begin{equation*}
[F,f] = \frac{1}{\sqrt{2}}
\begin{bmatrix} 
 & & -\left( f(0) - f(1) \right) & -\left( f(0) - f(3) \right) \\ 
 & & f(2) - f(1) & -\left( f(2) - f(3) \right) \\ 
f(0) - f(1) & -\left(f(2) - f(1) \right) & & \\ 
f(0) - f(3) & f(2) - f(3) & & 
\end{bmatrix}. 
\end{equation*}
Here, for the simplicity, we denote $i = v_{i}$ the 
vertices on the squares $V_{\bs{s}}$. 
We denote the upper right $2\times 2$ block of $[F,f]$ by $d^{-}f$ and the lower left by $d^{+}f$.


We now construct a sequence of operators on $\displaystyle \bigoplus_{\bs{s} \in S^{\times n}} \mathcal{H}_{\bs{s}}$ that will give rise to a cyclic cocycle. 
For $f,g,h \in C(K)$, 
we have
\begin{align*}
f[F,g][F,h] &= 
\begin{bmatrix} f^{+} & \\ & f^{-} \end{bmatrix} 
\begin{bmatrix} & d^{-}g \\ d^{+}g & \end{bmatrix} 
\begin{bmatrix} & d^{-}h \\ d^{+}h & \end{bmatrix} \\ 
&= 
\begin{bmatrix} & f^{+}d^{-}g \\ f^{-}d^{+}g & \end{bmatrix} 
\begin{bmatrix} & d^{-}h \\ d^{+}h & \end{bmatrix} 
= 
\begin{bmatrix} 
f^{+}d^{-}gd^{+}h &  \\ 
& f^{-}d^{+}gd^{-}h
\end{bmatrix}.
\end{align*}
By setting $f_{i,j} = f(j) - f(i)$, the two diagonal $2 \times 2$ blocks of $f[F,g][F,h]$ can be expressed as 
\begin{align*}
f^{+}d^{-}gd^{+}h 
&= \frac{1}{2}
\begin{bmatrix}  
f(0) & \\ & f(2) 
\end{bmatrix}
\begin{bmatrix}
g_{0,1} & g_{0,3} \\ -g_{2,1} & g_{2,3} 
\end{bmatrix}
\begin{bmatrix} 
-h_{0,1} & h_{2,1} \\ -h_{0,3} & -h_{2,3} 
\end{bmatrix} \\ 
&= \frac{1}{2}
\begin{bmatrix} 
f(0)g_{0,1} & f(0)g_{0,3} \\ -f(2)g_{2,1} & f(2)g_{2,3} 
\end{bmatrix}
\begin{bmatrix} 
-h_{0,1} & h_{2,1} \\ -h_{0,3} & -h_{2,3} 
\end{bmatrix} \\ 
&= \frac{1}{2} 
\begin{bmatrix} 
-f(0)(g_{0,1}h_{0,1} + g_{0,3}h_{0,3}) & f(0)(g_{0,1}h_{2,1} - g_{0,3}h_{2,3}) \\ 
f(2)(g_{2,1}h_{0,1} - g_{2,3}h_{0,3}) & -f(2)(g_{2,1}h_{2,1} + g_{2,3}h_{2,3}) 
\end{bmatrix} \\ 
&= \frac{1}{2}
\begin{bmatrix} 
f(0)(g_{0,1}h_{1,0} + g_{0,3}h_{3,0}) & -f(0)(g_{0,1}h_{1,2} - g_{0,3}h_{3,2}) \\ 
f(2)( g_{2,3}h_{3,0} - g_{2,1}h_{1,0}) & f(2)(g_{2,1}h_{1,2} + g_{2,3}h_{3,2}) 
\end{bmatrix}  
\end{align*}
and
\begin{align*}
f^{-}d^{+}gd^{-}h 
&= \frac{1}{2} 
\begin{bmatrix}
f(1) & \\ & f(3) 
\end{bmatrix}
\begin{bmatrix}
-g_{0,1} & g_{2,1} \\ 
-g_{0,3} & -g_{2,3}  
\end{bmatrix}
\begin{bmatrix} 
h_{0,1} & h_{0,3} \\ 
-h_{2,1} & h_{2,3}
\end{bmatrix} \\ 
&= \frac{1}{2}
\begin{bmatrix} 
-f(1)g_{0,1} & f(1)g_{2,1} \\ 
-f(3)g_{0,3} & -f(3)g_{2,3} 
\end{bmatrix}
\begin{bmatrix} 
h_{0,1} & h_{0,3} \\ 
-h_{2,1} & h_{2,3}
\end{bmatrix} \\ 
&= \frac{1}{2} 
\begin{bmatrix} 
f(1)(-g_{0,1}h_{0,1} - g_{2,1}h_{2,1}) & f(1)(-g_{0,1}h_{0,3} + g_{2,1}h_{2,3}) \\ 
f(3)(-g_{0,3}h_{0,1} + g_{2,3}h_{2,1}) & f(3)(-g_{0,3}h_{0,3} - g_{2,3}h_{2,3}) 
\end{bmatrix} \\ 
&= \frac{1}{2} 
\begin{bmatrix} 
f(1)(g_{1,0}h_{0,1} + g_{1,2}h_{2,1}) & f(1)(g_{1,0}h_{0,3} - g_{1,2}h_{2,3}) \\ 
-f(3)(g_{3,2}h_{2,1} - g_{3,0}h_{0,1} ) & f(3)(g_{3,0}h_{0,3} + g_{3,2}h_{2,3}) 
\end{bmatrix}.
\end{align*}
Here, we define
\begin{equation*}
M = -\frac{2}{e^{2}}[F,x][F,y] 
= 
\begin{bmatrix} 
 & 1 & & \\ 
-1 & & & \\ 
 & & & 1 \\ 
 & & -1 &  
\end{bmatrix}.
\end{equation*}
$M$ can be expressed as $N \oplus N$ by denoting the off-diagonal $2\times2$ matrices by $N$. 
The following lemma indicates that 
the trace of an operator $f[F,g][F,h]M$ 
gives rise to a discretized version of 
integration on a square.

\begin{lemma}
\label{lem:rsint}
For any $n \in \mathbb{N}$ and $\bs{s} \in S^{\times n}$, 
we have 
\begin{align*}
2 \cdot \operatorname{Tr}(f[F,g][F,h]M) =& f(0)(g_{0,1}h_{1,2} - g_{0,3}h_{3,2})
+ f(2)(g_{2,3}h_{3,0} - g_{2,1}h_{1,0}) \\
\ &- f(1)(g_{1,0}h_{0,3} - g_{1,2}h_{2,3}) 
- f(3)( g_{3,2}h_{2,1} - g_{3,0}h_{0,1}).
\end{align*}
\end{lemma}

\begin{proof} The proof follows straightforward calculation. By multiplying $M$ with $f[F,g][F,h]$, we have 
\begin{align*}
f[F,g][F,h]M  
&= 
\begin{bmatrix} 
f^{+}d^{-}gd^{+}h &  \\ 
& -f^{-}d^{+}gd^{-}h 
\end{bmatrix}
\begin{bmatrix} 
N & \\ & N 
\end{bmatrix} \\ 
&= 
\begin{bmatrix} 
f^{+}d^{-}gd^{+}hN &  \\ 
& f^{-}d^{+}gd^{-}h \; N
\end{bmatrix}.
\end{align*}
Then we have 
\begin{align*}
f^{+}d^{-}gd^{+}hN 
&= \frac{1}{2}
\begin{bmatrix}
f(0)(g_{0,1}h_{1,0} + g_{0,3}h_{3,0}) & -f(0)(g_{0,1}h_{1,2} - g_{0,3}h_{3,2}) \\ 
f(2)( g_{2,3}h_{3,0} - g_{2,1}h_{1,0}) & f(2)(g_{2,1}h_{1,2} + g_{2,3}h_{3,2}) 
\end{bmatrix}
\begin{bmatrix}
 & 1 \\ -1 & 
\end{bmatrix} \\ 
&= \frac{1}{2}
\begin{bmatrix} 
f(0)(g_{0,1}h_{1,2} - g_{0,3}h_{3,2}) & f(0)(g_{0,1}h_{1,0} + g_{0,3}h_{3,0}) \\ 
-f(2)(g_{2,1}h_{1,2} + g_{2,3}h_{3,2}) & f(2)(g_{2,3}h_{3,0} - g_{2,1}h_{1,0} )
\end{bmatrix}  
\end{align*} 
and 
\begin{align*}
f^{-}d^{+}gd^{-}h \; N 
&= \frac{1}{2}
\begin{bmatrix} 
f(1)(g_{1,0}h_{0,1} + g_{1,2}h_{2,1}) & f(1)(g_{1,0}h_{0,3} - g_{1,2}h_{2,3}) \\ 
-f(3)( g_{3,2}h_{2,1} - g_{3,0}h_{0,1}) & f(3)(g_{3,0}h_{0,3} + g_{3,2}h_{2,3}) 
\end{bmatrix} 
\begin{bmatrix} 
 & 1 \\ -1 & 
\end{bmatrix} \\ 
&= \frac{1}{2} 
\begin{bmatrix} 
-f(1)(g_{1,0}h_{0,3} - g_{1,2}h_{2,3}) & f(1)(g_{1,0}h_{0,1} + g_{1,2}h_{2,1}) \\ 
-f(3)(g_{3,0}h_{0,3} + g_{3,2}h_{2,3}) & -f(3)(g_{3,2}h_{2,1} - g_{3,0}h_{0,1} )
\end{bmatrix}.
\end{align*}
Therefore the trace of $f[F,g][F,h]M$ is 
given by the lemma. 
\end{proof}



When $g = x$ and $h = y$ are the coordinate functions of $\mathbb{R}^{2}$, 
respectively, each term of the right hand side of Lemma \ref{lem:rsint} 
is nothing but the summand of the Riemannian sum of $f$. 
So the sum of the trace in the left hand side can be regard as an analogue of 
the Riemannian sum on the Cantor dust.

\begin{definition}
\label{def:phi_n}
For any $n \in \mathbb{N}$ and 
$f \otimes g \otimes h \in C(CD)\otimes C(CD) \otimes C(CD)$, define a map
\[
\phi_{n}(f,g,h) = \sum_{\bs{s} \in S^{\times n}}\operatorname{Tr}(f[F,g][F,h]M). 
\]
We call the sequence of the maps 
$\{ \phi_{n} \}$
the approximating combinatorial integration.
\end{definition}


\begin{proposition}
\label{prop:Lipschitz}
For any functions $f \in C(CD)$ and $g,h \in C^{\mathrm{Lip}}(CD)$, 
the sequence $\{ \phi_{n}(f,g,h) \}$ 
converges to $0$. 
\end{proposition}

\begin{proof} 
By Lemma \ref{lem:rsint}, we have  
\begin{align*}
\bigl| \phi_{n}(f,g,h) \bigr|
&\leq 
\sum_{\bs{s} \in S^{\times n}}
\Bigl( 
\bigl| 
f(0)(g_{0,1}h_{1,2} - g_{0,3}h_{3,2})
\bigr|
+ 
\bigl|
f(2)( g_{2,3}h_{3,0} - g_{2,1}h_{1,0})
\bigr| \\
&\hspace*{15mm}+  
\bigl|
f(1)(g_{1,0}h_{0,3} - g_{1,2}h_{2,3}) 
\bigr|
+ 
\bigl|
f(3)( g_{3,2}h_{2,1} - g_{3,0}h_{0,1}) 
\bigr| \Bigr) \\ 
&\leq 
2\mathrm{Lip}(g)\mathrm{Lip}(h)\dfrac{1}{9^{n}}\sum_{\bs{s} \in S^{\times n}}
\left( |f(0)| + |f(1)| + |f(2)| + |f(3)| \right) \\ 
&\leq 
8 \| f \|\mathrm{Lip}(g)\mathrm{Lip}(h)\dfrac{4^{n}}{9^{n}} 
\end{align*} 
for any $n \in \mathbb{N}$, 
$f \in C(CD)$, 
and 
$g,h \in C^{\mathrm{Lip}}(CD)$. 
Here, $\| f \|$ means the supnorm of $f$ 
and $\mathrm{Lip}(g)$ means 
the Lipschitz constant of $g \in C^{\mathrm{Lip}}(CD)$. 
Therefore, we have 
$\phi_{n}(f,g,h) \to 0$ ($n \to \infty$). 
\end{proof} 

By Proposition \ref{prop:Lipschitz}, 
our approximating combinatorial integration does not restore 
rich structure on the Lipschitz functions.  
So we need anothor class of functions on the Cantor dust. 

\subsection{Non-triviality of a combinatorial integration}
\label{subsec:phantomtorus}

Let $CD_{n} = \bigcup_{\bs{s} \in S^{\times n}}f_{\bs{s}}(I)$ 
be a space obtained by applying the $n$-times composition of $\displaystyle \bigcup_{s=1}^{4} f_{s}$ to $I$. 
$CD_{n}$ has $4^{n}$ distinct squares $\{\gamma^{n}_{i}\}_{1\leq i \leq4^{n}}$ with length of edge $\displaystyle \frac{1}{3^{n}}$. 
For every $n \in \mathbb{N}$, 
the restriction of $\bs{c}_{n}$ to $CD_{n}$ is still surjective. 
%
The image of $\{\gamma^{n}_{i}\}_{1\leq i \leq4^{n}}$ by $\bs{c}_{n}$ is then a subdivision of $I$ each of whose cell is a square with length $\displaystyle \frac{1}{2^{n}}$. 
Therefore, the maximum length across all the cells tends to $0$ as $n \to \infty$. 


We introduce an equivalence relation $\sim$ on $\partial I$: $(a,0) \sim (a,1)$ and $(0,b) \sim (1,b)$. The quotient space turns out to be the torus $\mathbb{T}^{2} = I/\sim$. 
The map 
$\bs{c}_{n}$ 
induces a surjective map 
$CD_{n} \to \mathbb{T}^{2}$.  
We denote the map by the same letter $\bs{c}_{n}$. 
Similarly, we can also construct 
a continuous function $\bs{c}: CD \rightarrow \mathbb{T}^{2}$  
by the restriction of the 2-dimensional Cantor dust function to $CD$. 
The following theorem shows that the limit of approximating combinatorial integration with nontrivial value 
exists for some class of function on $CD$. 


\begin{theorem}
\label{thm:index_thm}
Let $C^{\infty}(\mathbb{T}^2)$ be the smooth functions on the torus $\mathbb{T}^{2}$. For $f =\bs{c}^{*}(\tilde{f}), g = \bs{c}^{*}(\tilde{g}), h = \bs{c}^{*}(\tilde{h}) \in \bs{c}^{*}C^{\infty}(\mathbb{T}^{2})$, 
we have 
\[ 
\lim_{n \to \infty} 
\phi_{n}(f,g,h) = 2\int_{\mathbb{T}^{2}} \tilde{f}d\tilde{g} \wedge d\tilde{h}.
\]
Moreover, the limit only depends on $f,g,h \in \bs{c}^{*}C^{\infty}(\mathbb{T}^{2})$. 
Therefore, the functional $\displaystyle \phi(f,g,h) 
= 
\lim_{n \to \infty}  
\phi_{n}(f,g,h)$  is a cyclic $2$-cocycle on 
$\bs{c}^{*}C^{\infty}(\mathbb{T}^{2})$. 
\end{theorem} 

\begin{proof}
Let $\boxplus_{n}$ be the set of equilateral cells that consist of $2^{n}$-fold subdivision of $I$. 
By the first-order approximation, 
for $\tilde{g}, \tilde{h} \in C^{\infty}(\mathbb{T}^{2})$ we have 
\begin{align*} 
\tilde{g}_{0,1} 
&= \tilde{g}(1) - \tilde{g}(0)
= 2^{-n}\tilde{g}_{x}(0) + o(2^{-n}) \quad (n \to \infty), \\ 
\tilde{h}_{1,2} 
&= \tilde{h}(2) - \tilde{h}(1) 
= 2^{-n} \tilde{h}_{y}(0)  + o(2^{-n}) \quad (n \to \infty), \\ 
\tilde{g}_{0,3} 
&= \tilde{g}(3) - \tilde{g}(0) 
= 2^{-n}\tilde{g}_{y}(0)  + o(2^{-n}) \quad (n \to \infty), \\ 
\tilde{h}_{3,2} 
&= \tilde{h}(2) - \tilde{h}(3) 
= 2^{-n}\tilde{h}_{x}(0)  + o(2^{-n}) \quad (n \to \infty). 
\end{align*}  
Apply the above equations to $\tilde{f}(0)(\tilde{g}_{0,1}\tilde{h}_{1,2} - \tilde{g}_{0,3}\tilde{h}_{3,2})$ in Lemma \ref{lem:rsint}, and we get 
\[
\tilde{f}(0)(\tilde{g}_{0,1}\tilde{h}_{1,2} - \tilde{g}_{0,3}\tilde{h}_{3,2}) 
= 
4^{-n}\tilde{h}(0)(\tilde{g}_{x}(0)\tilde{h}_{y}(0) - \tilde{g}_{y}(0)\tilde{h}_{x}(0)) 
+ o(4^{-n}) \quad (n \to \infty).
\]
Therefore, we have 
\begin{align*}
&\phantom{=}\lim_{n \to \infty} 
\sum_{\square \in \boxplus_{n}} 
\tilde{h}(0)(\tilde{g}_{0,1}\tilde{h}_{1,2} - \tilde{g}_{0,3}\tilde{h}_{3,2}) \\ 
&= \lim_{n \to \infty} 4^{-n} \sum_{\square \in \boxplus_{n}}  
\tilde{h}(0)(\tilde{g}_{x}(0)\tilde{h}_{y}(0) - \tilde{g}_{y}(0)\tilde{h}_{x}(0)) 
+ 4^{n}o(4^{-n})\\
&= 
\int_{\mathbb{T}^{2}}\tilde{f}(\tilde{g}_{x} \tilde{h}_{y} - \tilde{g}_{y}\tilde{h}_{x})dx \wedge dy\\
&= 
\int_{\mathbb{T}^{2}}\tilde{f}d\tilde{g} \wedge d\tilde{h} .
\end{align*}
Similary, we obtain  
\begin{align*} 
\lim_{n \to \infty} 
\sum_{\square \in \boxplus_{n}}\tilde{f}(2)(   \tilde{g}_{2,3}\tilde{h}_{3,0} - \tilde{g}_{2,1}\tilde{h}_{1,0}) 
&= 
-\lim_{n \to \infty} 
\sum_{\square \in \boxplus_{n}}\tilde{f}(1)(\tilde{g}_{1,0}\tilde{h}_{0,3} - \tilde{g}_{1,2}\tilde{h}_{2,3}) \\ 
= - \lim_{n \to \infty} 
\sum_{\square \in \boxplus_{n}}\tilde{f}(3)(   \tilde{g}_{3,2}\tilde{h}_{2,1} - \tilde{g}_{3,0}\tilde{h}_{0,1}) 
&= 
\int_{\mathbb{T}^{2}}\tilde{f}d\tilde{g} \wedge d\tilde{h} .
\end{align*} 
Hence, by Lemma \ref{lem:rsint}, we get 
\begin{align*}
\lim_{n \to \infty} \phi_{n}(f,g,h)
&= \frac{1}{2} \lim_{n \rightarrow \infty} \sum_{\bs{s} \in S^{\times n}} \operatorname{Tr}\left( \tilde{f} \circ \bs{c} \Bigl[ F,\tilde{g} \circ \bs{c} \Bigr] \left[ F,\tilde{h} \circ \bs{c}\right] M \right) \\
&= \frac{1}{2} \lim_{n \rightarrow \infty} \sum_{\square \in \boxplus_{n}} \operatorname{Tr}\left( \tilde{f}[F,\tilde{g}][F,\tilde{h}]M \right) \\
&= \frac{1}{2} \lim_{n \rightarrow \infty} \sum_{\square \in \boxplus_{n}}
\left( \tilde{f}(0)(\tilde{g}_{0,1}\tilde{h}_{1,2} - \tilde{g}_{0,3}\tilde{h}_{3,2}) \right. \\
&\hspace*{24mm} + \tilde{f}(2)(\tilde{g}_{2,3}\tilde{h}_{3,0} - \tilde{g}_{2,1}\tilde{h}_{1,0}) \\
&\hspace*{24mm} - \tilde{f}(1)(\tilde{g}_{1,0}\tilde{h}_{0,3} - \tilde{g}_{1,2}\tilde{h}_{2,3}) \\
&\hspace*{24mm} \left. - \tilde{f}(3)( \tilde{g}_{3,2}\tilde{h}_{2,1} - \tilde{g}_{3,0}\tilde{h}_{0,1} ) \right) \\
&= 2 \int_{\mathbb{T}^{2}} \tilde{f}d\tilde{g} \wedge d\tilde{h}.
\end{align*}

Assume that $f = \bs{c}^{*}(\tilde{f}) = \bs{c}^{*}(\tilde{f}^{\prime})$, 
$g = \bs{c}^{*}(\tilde{g}) = \bs{c}^{*}(\tilde{g}^{\prime})$ and 
$h = \bs{c}^{*}(\tilde{h}) = \bs{c}^{*}(\tilde{h}^{\prime})$. 
In general, if $\bs{c}^{*}(\tilde{p}) = \bs{c}^{*}(\tilde{q})$ for $\tilde{p}, \tilde{q} \in C(\mathbb{T}^{2})$, then $\bs{c}^{*}_{n}(\tilde{p}) = \bs{c}^{*}_{n}(\tilde{q})$ on the boundary of $CD_{n}$ for any $n \in \mathbb{N}$. Therefore, for every $n \in \mathbb{N}$, we have 
\begin{align*}
\sum_{\bs{s} \in S^{\times n}} \operatorname{Tr}\left( \tilde{f} \circ \bs{c} \Bigl[ F,\tilde{g} \circ \bs{c} \Bigr] \left[ F,\tilde{h} \circ \bs{c}\right] M \right)
& = 
\sum_{\bs{s} \in S^{\times n}} \operatorname{Tr}\left( \tilde{f}^{\prime} \circ \bs{c} \Bigl[ F,\tilde{g}^{\prime} \circ \bs{c} \Bigr] \left[ F,\tilde{h}^{\prime} \circ \bs{c}\right] M \right)
\end{align*} 
Thus as $n \to \infty$, 
the value $\phi (f,g,h)$ depends only on 
$f,g,h \in \bs{c}^{*}C^{\infty}(\mathbb{T}^{2}) $. 
\end{proof}


\begin{corollary}
\label{cor:pairing}
Given $[p] \in K_{0}(\bs{c}^{*}(C^{\infty}(\mathbb{T}^{2})))$ written as $p = e \circ \bs{c}: CD  \rightarrow M_{N}(\mathbb{C})$, 
the Connes' pairing of the cyclic $2$-cocycle $[\phi]$ 
with $[p]$ is expressed as follows:  
\[
\langle [\phi], [p] \rangle = \dfrac{1}{\pi i} \int_{\mathbb{T}^{2}} \operatorname{Tr}(e(de)^{2}).
\]
\end{corollary}

\vspace*{1\baselineskip}

\subsection*{Acknowledgments} Seto was supported by JSPS KAKENHI Grant Number  21K13795.

\bibliographystyle{plain}
\bibliography{ref_applications}

\begin{thebibliography}{1}

\bibitem{MR1545974}
R.~Baer.
\newblock Abelian groups without elements of finite order.
\newblock {\em Duke Math. J.}, 3(1):68--122, 1937.

\bibitem{MR823176}
A.~Connes.
\newblock Noncommutative differential geometry.
\newblock {\em Inst. Hautes \'Etudes Sci. Publ. Math.}, 62:257--360, 1985.

\bibitem{arXiv:1912.05832}
T.~Maruyama and T.~Seto.
\newblock A combinatorial fredholm module on self-similar sets built on
  $n$-cubes, 2019.
\newblock arXiv:1912.05832.

\bibitem{talk:Moriyoshi2013}
H.~Moriyoshi.
\newblock On a cyclic volume cocycle in fractal geometry.
\newblock talk on ``Further developement of Atiyah-Singer Index Theory'', 2013
  (based on joint work with T. Natsume).

\end{thebibliography}

\end{document}